\documentclass[12pt]{article}
\usepackage{amsmath,amsfonts,amssymb,amsthm}
\usepackage{bbm,mathrsfs,bm,mathtools}
\title{Weighted cone-volume measures of pseudo-cones}
\author{Rolf Schneider}
\date{}
\newcommand{\Sn}{{\mathbb S}^{n-1}}

\newcommand{\R}{{\mathbb R}}

\newcommand{\C}{{\mathcal C}}

\newcommand{\K}{{\mathcal K}}

\newcommand{\Rn}{{\mathbb R}^n}

\newcommand{\N}{{\mathbb N}}

\newcommand{\Ha}{\mathcal{H}}

\newcommand{\B}{\mathcal{B}}
\newcommand{\D}{{\rm d}}

  \renewcommand{\exp}{{\rm exp}\,}

  \newcommand{\fed}{\,\rule{.1mm}{.20cm}\rule{.20cm}{.1mm}\,}

\newtheorem{theorem}{Theorem}
\newtheorem{lemma}{Lemma}
\newtheorem{definition}{Definition}

\begin{document}
\maketitle

\begin{abstract}
A pseudo-cone in $\R^n$ is a nonempty closed convex set $K$ not containing the origin and such that $\lambda K \subseteq K$ for all $\lambda\ge 1$. It is called a $C$-pseudo-cone if $C$ is its recession cone, where $C$ is a pointed closed convex cone with interior points. The cone-volume measure of a pseudo-cone can be defined similarly as for convex bodies, but it may be infinite. After proving a necessary condition for cone-volume measures of $C$-pseudo-cones, we introduce suitable weights for cone-volume measures, yielding finite measures. Then we provide a necessary and sufficient condition for a Borel measure on the unit sphere to be the weighted cone-volume measure of some $C$-pseudo-cone.\\[2mm]
\end{abstract}

\section{Introduction}\label{sec1}

The classical theory of convex bodies deals mainly with properties that are invariant under translations of the bodies. This changed with the seminal papers \cite{Lut75, Lut93, HLYZ16} on $L_p$ and dual Brunn--Minkowski theories, where the family $\K^n_{(o)}$ of convex bodies in $\R^n$ containing the origin $o$ in the interior is in the foreground. These theories increased considerably the challenge of various Minkowski type problems. Many questions on the existence of convex bodies with given measures of various types have been solved, and several others are open. Let us briefly recall the classical issue.

For a convex body $K$ (a compact, convex set with interior points) in Euclidean space $\R^n$, the surface area measure $S_{n-1}(K,\cdot)$ is defined as the image measure of $\Ha^{n-1}$ (where $\Ha^k$ denotes the $k$-dimensional Hausdorff measure) under the Gauss map of $K$. Thus, it is a finite Borel measure on the unit sphere $\Sn$ of $\R^n$. Minkowski's existence theorem provides necessary and sufficient conditions for a Borel measure on $\Sn$ to be the surface area measure of a convex body. If such a body exists, it is unique up to a translation. For references and further information, we refer to \cite[Sect. 8.2]{Sch14} and its Notes. 

The cone-volume measure $V_K$ of a convex body $K\in\K^n_{(o)}$ is also a Borel measure on $\Sn$. It can be defined as follows. Let $\nu_K(x)$ be the outer unit normal vector of $K$ at points $x\in \partial K$ where it is unique (which is the case $\Ha^{n-1}$-almost everywhere on $\partial K$, the boundary of $K$). Then
\begin{equation}\label{1.1}  
V_K(\omega)=\frac{1}{n}\int_{\nu_K^{-1}(\omega)} \langle x,\nu_K(x)\rangle\,\Ha^{n-1}(\D x)= \frac{1}{n} \int_\omega h_K(u)\,\D S_{n-1}(K,u)
\end{equation}
for Borel sets $\omega\subset\Sn$, where $\langle\cdot\,,\cdot\rangle$ denotes the scalar product of $\R^n$ and  $h_K$ is the support function of $K$. The name `cone-volume measure' comes from the fact that 
\begin{equation}\label{1.2} 
V_K(\omega)= \Ha^n\left(\bigcup_{x\in \nu_K^{-1}(\omega)} [o,x]\right),
\end{equation}
where $[o,x]$ is the closed segment with endpoints $o$ and $x$. For the equivalence of these representations, we refer to \cite[p. 501]{Sch14}.

The Minkowski type problem for cone-volume measures asks for necessary and sufficient conditions for a Borel measure on $\Sn$ to be the cone-volume measure of some convex body. It is also known as the $L_0$-Minkowski problem or the logarithmic Minkowski problem. In the case of existence, also uniqueness is of interest. Early contributions to the two-dimensional case are due to Gage and Li \cite{GL94} and to  Stancu \cite{Sta02, Sta03}. In $n$ dimensions, B\"or\"oczky, Lutwak, Yang and Zhang \cite{BLYZ13} discovered that for even measures the so-called `subspace concentration condition' is necessary and sufficient. Treatments of the discrete case were provided by Zhu \cite{Zhu14} and by B\"or\"oczky, Heged\H{u}s and Zhu \cite{BHZ16}. That the subspace concentration condition is sufficient also for non-symmetric measures was proved by Chen, Li and Zhu \cite{CLZ19}. Further investigations on cone-volume measures are \cite{HL14, BH15, BH16b, BH17}.

In the last decade, it has become clear that a counterpart to the set $\K^n_{(o)}$ of convex bodies containing the origin in the interior can be seen in the set of pseudo-cones. A pseudo-cone in $\R^n$ is a nonempty closed convex set $K$ not containing the origin and satisfying $\lambda K\subseteq K$ for $\lambda\ge 1$. Thus, pseudo-cones are special unbounded convex sets. If $C$ is the recession cone of the pseudo-cone $K$ (i.e., $C=\{z\in\R^n: K+\lambda z\subseteq K\,\forall \lambda\ge 0\}$), we have $K\subset C$. In the following, we assume always that a pointed closed convex cone $C\subset\R^n$ with interior points is given, and we will be interested in the pseudo-cones with recession cone $C$, called $C$-pseudo-cones. The set of all $C$-pseudo-cones in $\R^n$ is denoted by $ps(C)$. 

A pseudo-cone $K\in ps(C)$ is called $C$-full if $C\setminus K$ is bounded, and $C$-close if $C\setminus K$ has finite volume (called the covolume of $K$). Khovanski\u{\i} and Timorin \cite{KT14} introduced $C$-full sets (under a different name) and proved complemented versions of the Aleksandrov--Fenchel, Brunn--Minkowski and Minkowski inequalities for them. A Brunn--Minkowski theory for $C$-close pseudo-cones was developed in \cite{Sch18}. The similarity of $ps(C)$ to $\K^n_{(o)}$ became even more evident when the existence of a counterpart to polarity, called copolarity, was noticed. In a special case it was introduced by Rashkovskii \cite{Ras17}. It appeared within a general study of dualities by Artstein--Avidan, Sadovsky and Wyczesany \cite{ASW22} and was further studied by Xu, Li and Leng \cite{XLL22} and by Schneider \cite{Sch24a}.

Surface area measures and cone-volume measures of pseudo-cones can essentially be defined as for convex bodies. We denote by $C^\circ=\{x\in\R^n: \langle x,y\rangle\le 0\,\forall y\in C\}$ the dual cone of $C$ and need the open subsets of the unit sphere $\Sn$ given by
$$ \Omega_C:= \Sn\cap {\rm int}\,C,\qquad \Omega_{C^\circ}:= \Sn\cap{\rm int}\,C^\circ.$$
Here ${\rm int}$ denotes the interior. The surface area measure of $K\in ps(C)$ is defined by 
$$ S_{n-1}(K,\omega) := \Ha^{n-1}(\nu_K^{-1}(\omega))\quad \mbox{for } \omega\in\B(\Omega_{C^\circ}),$$
where $\B(X)$ denotes the $\sigma$-algebra of Borel sets of a topological space $X$. The cone-volume measure $V_K$ of $K$ is defined by (\ref{1.2}), but now only for Borel sets $\omega\subset\Omega_{C^\circ}$.

The mappings $S_{n-1}(K,\cdot)$ and $V_K$ are again Borel measures, but in contrast to the case of convex bodies they are only defined on $\Omega_{C^\circ}$, and they can be infinite, though they are finite on compact sets. Each of these measures can also be the zero measure, namely if $K=F+C$ with a nonempty compact convex set $F\subset C$ of dimension less than $n-1$. 
 
The Minkowski problem for pseudo-cones asks for necessary and sufficient conditions for a Borel measure $\varphi$ on $\Omega_{C^\circ}$ in order that there is a pseudo-cone $K\in ps(C)$ with $S_{n-1}(K,\cdot)=\varphi$. A general answer is unknown. That finiteness and compact support are sufficient was proved in \cite{Sch18}. Extensions of this result to $L_p$ versions were treated by Yang, Ye and Zhu \cite{YYZ22}, and versions for dual curvature measures by Li, Ye and Zhu \cite{LYZ23}. A further extension, to so-called $(p,q)$-dual curvature measures, is due to Ai, Yang and Ye \cite{AYZ24}.

It has repeatedly been pointed out to the author, so by Jacopo Ulivelli, Vadim Semenov and  Yiming Zhao, that the results from \cite{Sch18, Sch21} for finite measures can also be deduced from known results about Monge--Amp\`{e}re equations. A certain transformation is needed here. For this transformation we refer, for example, to Bakelman \cite[Sect. 4]{Bak83}, where also pseudo-cones appear implicitly, though not under this name. For weightings as treated in the following, a direct geometric approach seems to be more natural.

The Minkowski problem for cone-volume measures of pseudo-cones can be formulated similarly. It was proved in \cite[Thm. 5]{Sch18} that every nonzero finite Borel measure on $\Omega_{C^\circ}$ is the cone-volume measure of some $C$-pseudo-cone $K$ (but finiteness is not a necessary condition). That $K$ is uniquely determined was proved by Yang, Ye and Zhu \cite[Thm. 7.3]{YYZ22}. The fact that no analogue of the subspace concentration condition is required, makes the Minkowski problem for cone-volume measures of pseudo-cones easier than for convex bodies; on the other hand, infinite measures cause new problems.

When the finiteness condition is dropped, the mentioned Minkowski problems are open. For surface area measures, a necessary condition, in the form of a moderate growth condition close to the boundary of $\Omega_{C^\circ}$, was proved in \cite {Sch21}. A similar necessary condition for cone-volume measures is the subject of the following theorem. For sets $\omega\in\B(\Omega_{C^\circ})$ we denote by $\Delta(\omega)$ the spherical distance of $\omega$ from the boundary of $\Omega_{C^\circ}$, that is
$$ \Delta(\omega):= \inf\{\angle(u,v): u\in\omega,\,v\in\partial \Omega_{C^\circ}\}$$
(it is clear from the context whether the boundary operator $\partial$ (and later ${\rm int}$) refers to $\Rn$ or to $\Sn$). Here $\angle(u,v)$ denotes the angle between the vectors $u$ and $v$. According to \cite[Sect. 2]{Sch24a}), we can fix a unit vector $\mathfrak v$ with $\mathfrak v\in{\rm  int}\,C$ and $-\mathfrak v\in{\rm int}\,C^\circ$.

\begin{theorem}\label{T1.1n}
Let $K\in ps(C)$. There is a constant $c$, depending only on $C$, $\mathfrak v$ and $K$, such that
$$ V_K(\omega) \le \frac{c}{\Delta(\omega)^{n-1}}$$
for each compact set $\omega\subset\Omega_{C^\circ}$.
\end{theorem}

In general, necessary and sufficient conditions are unknown. However, and this leads us to the core of this investigation, we can look at surface area and cone-volume measures of $C$-pseudo-cones from a different point of view. For $K\in ps(C)$, the shape of $K$ is strongly influenced by the shape of $C$ if the distance from the origin tends to infinity. For this reason, we introduce weightings that become smaller where the measures under consideration are of lesser importance. For convex bodies, weighted Minkowski problems were considered, for example, by Livshyts \cite{Liv19} and by Kryvonos and Langharst \cite{KL23}. It seems that for pseudo-cones, suitable weightings are even more desirable, allowing us to neglect those regions where the measures loose interest. For $C$-pseudo-cones, weighted surface area measures were treated in \cite{Sch24c}. Underlying there was the following definition.

\begin{definition}\label{D1.1} 
Let $\Theta: C\setminus\{o\}\to(0,\infty)$ be continuous and homogeneous of degree $-q$, where $n-1<q<n$.
\end{definition}

The $\Theta$-weighted surface area measure of $K\in ps(C)$ is defined by 
$$ S_{n-1}^\Theta(K,\omega):= \int_{\nu_K^{-1}(\omega)} \Theta(y)\,\Ha^{n-1}(\D y)$$
for $\omega\in\B(\Omega_{C^\circ})$. The following was proved in \cite{Sch24c}. The measure $S_{n-1}^\Theta(K,\cdot)$ is finite. Given a finite Borel measure $\varphi$ on $\Omega_{C^\circ}$ with compact support, there exists a $C$-pseudo-cone $K$ with $S_{n-1}^\Theta(K,\cdot)=\varphi$.

It is the main purpose of this paper to obtain a corresponding result for cone-volume measures. Although we partially use similar arguments, there are some crucial differences. 

We use the density $\Theta$ to define the $n$-dimensional $\Theta$-weighted Hausdorff measure on ${\rm int}\,C$ by
\begin{equation}\label{1.0} 
\Ha_\Theta^n(\eta):= \int_\eta \Theta(y)\,\Ha^n(\D y),\quad \eta\in\B({\rm int}\,C).
\end{equation}
The $\Theta$-weighted cone-volume measure of $K\in ps(C)$ is then defined by
$$ V_K^\Theta(\omega) := \Ha_\Theta^n\left(\bigcup_{x\in \nu_K^{-1}(\omega)} [o,x]\right), \quad \omega\in \B(\Omega_{C^\circ}).$$
It is clear that $V_K^\Theta$ is a measure, and it will turn out that it is finite. 

\begin{theorem}\label{T1.2}
A Borel measure on $\Omega_{C^\circ}$ is the $\Theta$-weighted cone-volume measure of some $C$-pseudo-cone if and only if it is finite.
\end{theorem}

In the next section, we collect some notation and basic results. Theorem \ref{T1.1n} is shown in Section \ref{sec4}. Section \ref{sec3} proves some properties of the weighted cone-volume measures, in particular, their finiteness and a special weak continuity. The proof of Theorem \ref{T1.2} is then completed in Section \ref{sec5}, by showing that finiteness is sufficient.

\section{Preliminaries}\label{sec2}

We work in $n$-dimensional Euclidean space $\R^n$ ($n\ge 2$) with scalar product $\langle\cdot\,,\cdot\rangle$ (the induced norm is denoted by $\|\cdot\|$), origin $o$, unit ball $B^n$ and unit sphere $\Sn$. 
A pointed $n$-dimensional closed convex cone $C$ is given, and $C^\circ$ is its dual cone. A unit vector $\mathfrak v\in{\rm int}\,C\cap{\rm int}(-C^\circ)$ is fixed. 

Hyperplanes and closed halfspaces of $\R^n$ are written in the form
$$ H(u,t) =\{x\in\R^n:\langle x,u\rangle = t\},\qquad H^-(u,t) =\{x\in\R^n:\langle x,u\rangle\le t\}$$
for $u\in\Sn$ and $t\in\R$, and for $t>0$ we define the compact set
$$ C^-(t) := C\cap H^-(\mathfrak{v},t).$$

For $K_j\in ps(C)$, $j\in\N_0$, we write $K_j\to K_0$ as $j\to \infty$ if there exists $t_0>0$ with $K_j\cap C^-(t_0)\not=\emptyset$ for all $j$ and $K_j\cap C^-(t)\to K_0\cap C^-(t)$ for all $t>t_0$, in the sense of the ordinary convergence of convex bodies. When in the following continuity is mentioned for pseudo-cones, it refers to this convergence.

The support function of $K\in ps(C)$ is defined by
$$h_K(x):=\sup\{\langle x,y\rangle: y\in K\},\quad x\in C^\circ.$$
We note that the supremum is a maximum if $x\in{\rm int}\,C^\circ$ and that $h_K$ is bounded and non-positive (it is negative on ${\rm int}\,C$). For the latter reason, we also write 
$$\overline h_K:= -h_K.$$ 
The halfspace $H^-(u,h_K(u))$ is the supporting halfspace (containing $K$) of $K$ with outer unit normal vector $u\in\Omega_{C^\circ}$.

We define
$$ \varrho_K(v):= \min\{\lambda\in\R: \lambda v\in K\}\quad\mbox{for }v\in \Omega_C.$$
For $y\in{\rm int}\,C$ there exists $\lambda\in\R$ with $\lambda y\in K$ since $C$ is the recession cone of $K$, and the minimum exists since $K$ is closed. We have $\varrho_K(v)v\in\partial K$ for $v\in \Omega_C$, and we define the mapping $r_K:\Omega_C\to \partial K$ by
$$ r_K(v):=\varrho_K(v)v,\quad v\in \Omega_C.$$
Further, we denote by $\partial{\hspace{1pt}}'K$ the set of points $x\in \partial K\cap {\rm int}\,C$ where $\nu_K(x)$ is unique. We shall use repeatedly that $\nu_K$ is continuous on $\partial{\hspace{1pt}}'K$. We define
$$ \alpha_K:= \nu_K\circ  r_K \quad\mbox{on }r_K^{-1}(\partial{\hspace{1pt}}'K).$$
The continuity of the maps $\varrho_K$, $r_K$, $\alpha_K$ follows from well-known results for convex bodies. These maps are known, respectively, as the radial function, the radial map. and the radial Gauss map of $K$.

Let $\varpi\subset\Omega_{C^\circ}$ be nonempty and compact. A pseudo-cone $K\in ps(C)$ is said to be $C$-determined by $\varpi$ if
$$ K= C\cap \bigcap_{u\in \varpi} H^-(u, h_K(u)).$$
Such a set is $C$-full, since $\varpi$ is compact. The set of all $C$-pseudo-cones that are $C$-determined by $\varpi$ is denoted by $\K(C,\varpi)$.

We recall the definition of Wulff shapes in the framework of $C$-pseudo-cones (see \cite[Sect. 5]{Sch18}). If $\varpi\subset\Omega_{C^\circ}$ is a nonempty compact set and $f:\varpi\to(0,\infty)$ is a positive continuous function, we define
$$ [f]:= C\cap \bigcap_{u\in\varpi} \{y\in\R^n: \langle y,u\rangle \le -f(u)\}.$$
Then $[f]\in \K(C,\varpi)$, and we call $[f]$ the Wulff shape associated with $(C,\varpi,f)$. Clearly, $\overline h_{[f]}\ge f$ on $\varpi$, and for $K\in \K(C,\varpi)$ we have $[\overline h_K|_\varpi] =K$.

\section{Proof of Theorem \ref{T1.1n}}\label{sec4}

Theorem \ref{T1.1n} could be derived from a corresponding result for surface area measures, but we prefer to deduce it directly from the following lemma. This lemma was essentially obtained in \cite{Sch21}, but we give a simplified and clarified version of the argument leading to it, with a slightly more explicit estimate.

\begin{lemma}\label{L3.1}
Let $s> 0$. There is a constant $c_1$, depending only on $C$ and  $\mathfrak{v}$, with the following property. To each compact set $\omega\subset \Omega_{C^\circ}$, there exists a number $t_\omega$, satisfying
\begin{equation}\label{4a}
t_\omega\le \frac{c_1s}{\Delta(\omega)},
\end{equation}
such that $H(u,\tau)\cap C^-(s)\not=\emptyset$ with $u\in\omega$ implies $H(u,\tau)\cap C\subset C^-(t_\omega)$.
\end{lemma}

\begin{proof}
Let $\omega\subset\Omega_{C^\circ}$ be a compact subset. The assertion is trivial if $\omega=\emptyset$ or $\omega=\{-\mathfrak{v}\}$, so we exclude these cases. Define
$$ t_\omega:= \max\{\langle x,\mathfrak{v}\rangle: x\in C\cap H(u,\tau), \mbox{ where } u\in\omega,\, H(u,\tau)\cap C^-(s)\not=\emptyset\}.$$
The maximum exists, due to the compactness of $C^-(s)$ and $\omega$. Clearly, $t_\omega>s$. There is a hyperplane $ H(u',\tau')$ with $u'\in\omega$ and $H(u',\tau')\cap C^-(s)\not=\emptyset$, and there is a point $x\in H(u',\tau')\cap C$ such that $t_\omega=\langle x,\mathfrak{v}\rangle$. Necessarily, $x\in\partial C$. By the definition of $t_\omega$, if $H(u,\tau)$ is any hyperplane with $u\in\omega$ and $H(u,\tau)\cap C^-(s)\not=\emptyset$, then $H(u,\tau)\cap C\subset H^-(\mathfrak{v},t_\omega)$.

By the choice of the hyperplane $H(u',\tau')$, we can choose a point $y\in H(u',\tau')\cap\partial C\cap H(\mathfrak{v},s)$. The segment $[x,o]$ intersects the hyperplane $H(\mathfrak{v},s)$ in a point $z$. Let 
$$ a:=\|y-z\|,\quad b:=\|z-x\|,\quad c:=\|x-y\|.$$
For the triangle with vertices $x,y,z$, let $\alpha$ be the angle at $x$, $\beta$ the angle at $y$, and $\gamma$ the angle at $z$. Then
\begin{equation}\label{3.1} 
\alpha\ge \Delta(\omega).
\end{equation}
This holds since the segment $[x,z]$ lies in a supporting hyperplane of $C$ with an outer normal vector $v$ (hence $v\in\partial C^\circ$), the segment $[x,y]$ lies in a hyperplane with normal vector $u'\in\omega$, hence the angle $\alpha$ is at least the angle between $u'$ and $v$, which is at least $\Delta(\omega)$.

Let $h$ be the distance of $x$ from the hyperplane $H(\mathfrak{v},s)$; thus $t_\omega=h+s$. We have
\begin{equation}\label{3.1a}
\frac{c}{\sin\gamma}= \frac{a}{\sin\alpha} \quad\mbox{and}\quad \frac{h}{c}\le\sin\beta.
\end{equation}
For the latter inequality, we note that there is a point $z'$ on the line through $y$ and $z$ such that the triangle with vertices $x,y,z'$ has a right angle at $z'$, hence $\|x-z'\|/c=\sin\beta$. Since $z'\in H({\mathfrak v}, s)$, we have $h\le \|x-z'\|$. From (\ref{3.1a}) we obtain
\begin{equation}\label{3.2n}
h\le \frac{a}{\sin\alpha} \sin\gamma\sin\beta.
\end{equation}
There is a positive constant $c_2$ depending only on $C$ and $\mathfrak{v}$ such that $a\le c_2s$. Since $\mathfrak{v}\in{\rm int}\,C$, we have $\gamma>\pi/2$ and thus $\alpha<\pi/2$, hence $\sin\alpha\ge 2\alpha/\pi$. Therefore, it follows from (\ref{3.2n}) and (\ref{3.1})  that 
$$ h\le \frac{c_2s}{\sin\alpha}\le\frac{c_2s}{\Delta(\omega)}\cdot\frac{2}{\pi}$$ 
and hence also
$$ t_\omega\le \frac{c_1s}{\Delta(\omega)}$$
with a constant $c_1$ depending only on $C$ and $\mathfrak{v}$, as stated.
\end{proof}

\noindent{\em Proof of Theorem} \ref{T1.1n}. 
Let $K$ be a $C$-pseudo-cone. We define 
$$ s:= \min\{t\in\R:K\cap H(\mathfrak{v},t)\not=\emptyset\}.$$  
With this number $s$, we define $c_1$ as in Lemma \ref{L3.1}. Then $c_1$ depends only on $C$, $\mathfrak{v}$ and $K$. 

Let $\omega\subset\Omega_{C^\circ}$ be a compact subset. To this $\omega$, there is a number $t_\omega$ satisfying the properties listed in Lemma \ref{L3.1}. Let $x\in \nu_K^{-1}(\omega)$. Then $x$ is contained in a supporting hyperplane $H(u,\tau)$ of $K$ with $u\in\omega$, and this hyperplane intersects $C^-(s)$. It follows from the properties of $t_\omega$ that $x\in H^-(\mathfrak{v}, t_\omega)$. Thus, $\nu_K^{-1}(\omega)\subset C^-(t_\omega)$. Therefore, the set
\begin{equation}\label{3.b}
M(K,\omega) := \bigcup_{x\in\nu_K^{-1}(\omega)} [o,x]
\end{equation}
satisfies $M(K,\omega) \subset C^-(t_\omega)$.

Choose $z\in K\cap H(\mathfrak{v},s)$. Then $C_z:= C+z\subseteq K$ and hence 
\begin{eqnarray*}
V_K(\omega) &=& \Ha^n(M(K,\omega)) \le \Ha^n(C^-(t_\omega)\setminus C_z)\\
&=& \Ha^n(C^-(s)) + \int_s^{t_\omega}[\Ha^{n-1}(H(\mathfrak{v},\tau)\cap C)- \Ha^{n-1}(H(\mathfrak{v},\tau)\cap C_z)]\,\D\tau.
\end{eqnarray*}
Since $t_\omega\ge s$, we have 
$$ \Ha^n(C^-(s)) = \Ha^n(C^-(1))s^n\le[\Ha^n(C^-(1)s]t_\omega^{n-1} = c_3t_\omega^{n-1},$$
where $c_3$ depends only on $C$, $\mathfrak{v}$ and $K$.

With $\Ha^{n-1}(C\cap H(\mathfrak{v},1))=:c_4$ we get
$$ V_K(\omega) \le c_3t_\omega^{n-1} + c_4\int_0^{t_\omega} [\tau^{n-1}-(\tau-s)^{n-1}]\,\D\tau.$$
Using $s\le \tau$ we have $\tau^{n-1}-(\tau-s)^{n-1}\le c_5\tau^{n-2}$, where $c_5$ depends only on $n$ (and thus on $C$). This gives
$$ V_K(\omega) \le c_3t_\omega^{n-1}+\frac{c_4c_5}{n-1}t_\omega^{n-1} = c_6t_\omega^{n-1},$$
where $c_6$ depends only on $C$, $\mathfrak v$ and $K$. Now Lemma \ref{L3.1} finishes the proof. \hfill$\Box$

\section{Properties of the weighted cone-volume measures}\label{sec3}

The following is Lemma 6 from \cite{Sch24c}.

\begin{lemma}\label{L3.1w}
Let $\varpi\subset\Omega_{C^\circ}$ be a nonempty, compact subset, and let $K_j\in\K(\C,\varpi)$ for $j\in\N_0$. Then $K_j\to K_0$ as $j\to\infty$ implies the weak convergence $S_{n-1}^\Theta(K_j,\cdot)\stackrel{w}{\to} S_{n-1}(K_0,\cdot)$.
\end{lemma}

We need also Lemma 5 from \cite{Sch24c}. (In the derivation of \cite[Lem. 5]{Sch24c} from \cite[Lem. 4]{Sch24c}, we either apply Lemma 4  to the positive and negative part of $g$, or state Lemma 4 for $\Ha^{n-1}$-integrable functions $F$.) Here $\int\cdot\,\D v$ indicates integration with respect to spherical Lebesgue measure.

\begin{lemma}\label{L3.3}
Let $K\in ps(C)$. Let $\omega\subseteq\Omega_{C^\circ}$ be a nonempty Borel set and $g:\omega\to\R$ be bounded and measurable. Then
$$ \int_\omega g(u)\,S^\Theta_{n-1}(K,\D u) = \int_{\alpha_K^{-1}(\omega)} g(\alpha_K(v))\Theta(r_K(v))\frac{\varrho^{n-1}_K(v)}{|\langle v,\alpha_K(v)\rangle|}\,\D v.$$
\end{lemma}

We can deduce properties of the weighted cone-volume measures from those of the weighted surface area measures, by using the following relation between them. It is reminiscent of (\ref{1.1}).

\begin{lemma}\label{L3.5}
For $K\in ps(C)$ and $\omega\in\B(\Omega_{C^\circ})$ we have
$$ V_K^\Theta(\omega)= \frac{1}{n-q} \int_\omega \overline h_K\,\D S_{n-1}^\Theta(K,\cdot).$$
\end{lemma}

\begin{proof}
Let  $K\in ps(C)$ and $\omega\in\B(\Omega_{C^\circ})$. Let $M(K,\omega)$ be defined by (\ref{3.b}). We use polar coordinates and the homogeneity of $\Theta$ to obtain
\begin{eqnarray}
V_K^\Theta(\omega) &=& \int_{M(K,\omega)}  \Theta(y)\,\Ha^n(\D y)\nonumber\\
&=& \int_0^\infty \int_{\Sn} {\mathbbm 1}_{M(K,\omega)}(rv)\Theta(rv)r^{n-1}\,\D v\, \D r\nonumber\\
&=& \int_{r_K^{-1}(\nu_K^{-1}(\omega))}\Theta(v) \int_0^{\varrho_K(v)} r^{n-1-q}\,\D r\,\D v\nonumber\\
&=& \frac{1}{n-q} \int_{\alpha_K^{-1}(\omega)} \Theta(v)\varrho_K^{n-q}(v)\,\D v.\label{3.1x}
\end{eqnarray}
Lemma \ref{L3.3} with $g:= \overline h_K|_\omega$ gives
$$ \int_\omega \overline h_K(u)\,S_{n-1}^\Theta(K,\D u) = \int_{\alpha_K^{-1}(\omega)} \overline h_K(\alpha_K(v))\Theta(r_K(v))\frac{\varrho_K^{n-1}(v)}{|\langle v,\alpha_K(v)\rangle|}\,\D v.$$
Here 
\begin{equation}\label{3.3a}
\overline h_K(\alpha_K(v))=|\langle \varrho_K(v)v,\alpha_K(v)\rangle|
\end{equation}
and $\Theta(r_K(v)) =\varrho_K(v)^{-q}\Theta(v)$, therefore
$$ \int_\omega \overline h_K\,\D S_{n-1}^\Theta(K,\cdot) = \int_{\alpha_K^{-1}(\omega)} \Theta(v)\varrho_K^{n-q}(v)\,\D v =(n-q)V_K^\Theta(\omega)$$
by (\ref{3.1x}), as stated.
\end{proof}

We use this to derive the following two lemmas.

\begin{lemma}\label{L3.6a}
For $K\in ps(C)$, the $\Theta$-weighted cone-volume measure $V_K^\Theta$ is finite.
\end{lemma}

\begin{proof}
This follows from Lemma \ref{L3.5}, the boundedness of the support functions of $C$-pseudo-cones, and the finiteness of the $\Theta$-weighted surface area measure, which was proved in \cite{Sch24c}.
\end{proof}

\begin{lemma}\label{L3.6}
Let $\varpi\subset\Omega_{C^\circ}$ be a nonempty, compact subset, and let $K_j\in\K(C,\varpi)$ for $j\in\N_0$. Then $K_j\to K_0$ as $j\to\infty$ implies the weak convergence $V_{K_j}^\Theta \stackrel{w}{\to}  V_{K_0}^\Theta$.
\end{lemma}

\begin{proof}
Let $f:\Omega_{C^\circ}\to\R$ be a bounded, continuous function. It follows from Lemma \ref{L3.5} that
$$ \int_{\Omega_{C^\circ}} f\,\D V_{K_j}^\Theta =  \frac{1}{n-q} \int_{\Omega_{C^\circ}} f \overline h_{K_j}\,\D S_{n-1}^\Theta(K_j,\cdot).$$
As $j\to\infty$, we have $\overline h_{K_j}\to \overline h_{K_0}$ uniformly (as can be deduced from \cite[Lem. 1.8.14]{Sch14}), and $S_{n-1}^\Theta(K_j,\cdot) \stackrel{w}{\to}  S_{n-1}^\Theta(K_0,\cdot)$ by Lemma \ref{L3.1w}. 
Therefore,
$$ \int_{\Omega_{C^\circ}} f\,\D V_{K_j}^\Theta \to \int_{\Omega_{C^\circ}} f\,\D V_{K_0}^\Theta\quad\mbox{as } j\to \infty,$$
and the assertion follows.
\end{proof}

We shall need the following consequence of this lemma.

\begin{lemma}\label{L3.7}
Let $K_j\in ps(C)$ for $j\in\N_0$, and let $\omega\subset\Omega_{C^\circ}$ be nonempty and compact. Suppose that $t>0$ satisfies $\nu_{K_0}^{-1}(\omega)\subset {\rm int}\,C^-(t)$  and
\begin{equation}\label{3.7} 
\emptyset\not=K_j\cap C^-(t)\to K_0\cap C^-(t)\quad\mbox{as }j\to\infty.
\end{equation}
Then the restriction of $V_{K_j}^\Theta$ to $\omega$ converges weakly to the restriction of $V_{K_0}^\Theta$ to $\omega$.
\end{lemma}

\begin{proof}
Since $\nu_{K_0}^{-1}(\omega)\subset {\rm int}\,C^-(t)$ and $\omega$ is compact, we can choose a compact set $\varpi\subset\Omega_{C^\circ}$ with $\omega\subset{\rm int}\,\varpi$ and $\nu_{K_0}^{-1}(\varpi)\subset {\rm int}\,C^-(t)$. By the convergence (\ref{3.7}) we have $\nu_{K_j}^{-1}(\varpi)\subset {\rm int}\,C^-(t)$ for sufficiently large $j$, and after changing the notation we can assume that this holds for all $j$.

For $K\in ps(C)$ we define
$$ K^{(\varpi)}:= C\cap \bigcap_{u\in\varpi} H^-(u,h_K(u)),$$
so that $ K^{(\varpi)}\in\K(C,\varpi)$. We have
\begin{equation}\label{3.8}
\nu_{K^{(\varpi)}}^{-1}(\omega)= \nu_K^{-1}(\omega).
\end{equation}
By Lemma 2 of \cite{Sch24c} and its proof, the convergence (\ref{3.7}) implies that 
$$   K_j^{(\varpi)} \to  K_0^{(\varpi)}\quad\mbox{as }j\to\infty.$$
The assertion now follows from (\ref{3.8}) and Lemma \ref{L3.6}.
\end{proof}

\section{Sufficiency of finiteness}\label{sec5}

For $K\in ps(C)$, we define the {\em $\Theta$-weighted covolume} of $K$ by
$$ V_\Theta(K) := V_K^\Theta(\Omega_{C^\circ}) = \int_{C\setminus K} \Theta(y)\,\Ha^n(\D y).$$

To show that finiteness of $\varphi$ is a sufficient condition in Theorem \ref{T1.2}, we  apply first a variational argument, using Wulff shapes. We need Lemma 11 of \cite{Sch24c}, which we reformulate below. We use that $\overline h_K$ for $K\in ps(C)$ is bounded away from $0$ on a compact set $\varpi\subset\Omega_{C^\circ}$.

\begin{lemma}\label{L5.1}
Let $\emptyset\not=\varpi\subset\Omega_{C^\circ}$ be compact and $K\in\K(C,\varpi)$. Let $f:\varpi\to \R$ be continuous, and let $[\overline h_K|_\varpi+tf]$ be the Wulff shape associated with $(C,\varpi,\overline h_K|_\varpi+tf)$ for $|t|\le\varepsilon$, where $\varepsilon>0$ is so small that $\overline h_K|_\varpi+tf>0$. Then
$$ \lim_{t\to 0} \frac{V_\Theta([\overline h_K|_\varpi+tf]) -V_\Theta(K)}{t}= \int_\varpi f(u)\,S_{n-1}^\Theta(K,\D u).$$
\end{lemma}

To complete now the proof of Theorem \ref{T1.2}, we assume that a finite Borel measure $\varphi$ on $\Omega_{C^\circ}$ is given. If $\varphi$ is the zero measure, then $V_K^\Theta=\varphi$ can be satisfied if, for example, we choose $K= C+z$ with $z\in {\rm int}\,C$. Therefore, we assume in the following that $\varphi$ is not the zero measure.

The proof that $\varphi$ can be realized as $V_K^\Theta$ for some pseudo-cone $K\in ps(C)$ proceeds in two steps. The first step assumes that $\varphi$ has compact support $\omega$ and finds the solution $K\in\K(C,\omega)$ by solving an extremum problem and applying the variational Lemma \ref{L5.1}. The second step applies this to an increasing sequence $(\omega_j)_{j\in \N}$ of compact sets with union $\Omega_{C^\circ}$ and employs the Blaschke selection theorem and the choice of a diagonal sequence. 

The first step can be found in \cite[Sect. 9]{Sch18} for unweighted surface area measures and in \cite[Sect. 7]{Sch24c} for $\Theta$-weighted surface area measures. 
For unweighted cone-volume measures, both steps were carried out in \cite[Sects. 11, 12]{Sch18}, and the second step for unweighted surface area measures appears in \cite[Sect. 3]{Sch21}. Therefore, we need not reproduce the full proof here, but we give only a sketch and point out where different arguments are required. A major difference to \cite{Sch18} is the use of Lemma \ref{L5.1}. Aleksandrov's original variational lemma, after which it is modeled, used in its proof (which is reproduced in \cite[Sect. 7.5]{Sch14}) inequalities for mixed volumes, thus restricting its generalizability. Fortunately, Huang, Lutwak, Yang and Zhang \cite[Lem. 4.3]{HLYZ16} found a different approach, which allowed several extensions of the lemma.

We assume, first, that $\varphi$ is a nonzero finite Borel measure with support contained in the compact set $\varpi\subset\Omega_{C^\circ}$. Temporarily we assume that $\varphi(\varpi)=1$. The functional to be maximized is defined by
$$ \Phi(f):= V_\Theta([f])^{-1/(n-q)} \exp\int_\varpi \log f\,\D\varphi,\quad f\in\C^+(\varpi),$$
where $\C^+(\varpi)$ is the set of positive continuous functions on $\varpi$. This functional is continuous and homogeneous of degree zero, as follows from Lemma 5 in \cite{Sch18},  the properties of the function $\Theta$ and from $\varphi(\varpi)=1$. As in the proof of Theorem 4 in \cite{Sch18}, one shows that $\Phi$ attains a maximum on $\{\overline h_L|_\varpi:L\in \K(C,\varpi),\,V_\Theta(L)=1\}$ at $\overline h_{K_0}|_\varpi$ for some set $K_0\in\K(C,\varpi)$ with $V_\Theta(K_0)=1$, and that this is also the maximum of $\Phi$ on $\C^+(\varpi)$.

Let $f$ be a continuous function on $\varpi$. The function
$$ h_t:= \overline h_{K_0}|_\varpi +tf\overline h_{K_0}|_\varpi\quad\mbox{for } |t|\le\varepsilon$$
belongs to $\C^+(\varpi)$ if $\varepsilon>0$ is sufficiently small. Hence, the derivative of $\Phi(h_t)$ at $t=0$ (which exists by Lemma \ref{L5.1}) is equal to zero, which together with Lemma \ref{L5.1} and $V_\Theta([h_0])= V_\Theta(K_0)=1$ leads to 
$$ \left[-\frac{1}{n-q} \int_\varpi f\overline h_{K_0}\,\D S_{n-1}^\Theta(K_0,\cdot) + \int_\varpi f\,\D\varphi\right]\exp\int_\varpi \log \overline h_{K_0}\,\D\varphi=0.$$
In view of Lemma \ref{L3.5} this is equivalent to
$$ \int_\varpi f\,\D V_{K_0}^\Theta = \int_\varpi f\,\D \varphi.$$
Since this holds for all continuous functions $f$ on $\varpi$, we deduce that $V_{K_0}^\Theta= \varphi$.

If $\varphi(\varpi)=1$ is not satisfied, then a suitable dilate $K$ of $K_0$ satisfies $V_K^\Theta=\varphi$, since $V_K^\Theta$ is homogeneous in $K$ of degree $n-q\not= 0$.

For the second step, we assume now that $\varphi$ is a nonzero, finite Borel measure on $\Omega_{C^\circ}$. We choose a sequence $(\omega_j)_{j\in\N}$ of compact subsets of $\Omega_{C^\circ}$ with $\varphi(\omega_1)>0$, $\omega_j\subset{\rm int}\,\omega_{j+1}$ for $j\in\N$ and $\bigcup_{j\in\N} \omega_j=\Omega_{C^\circ}$. Then we define $\varphi_j(\omega):= \varphi(\omega\cap\omega_j)$ for $\omega\in\B(\Omega_{C^\circ})$ and $j\in \N$. As shown above, for each $j\in\N$ there exists $K_j\in\K(C,\omega_j)$ with $V_{K_j}^\Theta=\varphi_j$.

The proof can be completed as that of Theorem 5 in \cite{Sch18}, once we have shown that there are two constants $c_1,c_2>0$, independent of $j$, such that the distance $b(K)$ of $K$ from the origin satisfies 
\begin{equation}\label{5.1}
c_1< b(K_j) < c_2\quad\mbox{for all }j\in\N.
\end{equation}

For the left inequality of (\ref{5.1}), we use the following lemma. 

\begin{lemma}\label{L5.2}
Let $K\in ps(C)$ satisfy $b(K)=1$. Let $\omega\in\B(\Omega_{C^\circ})$ be a Borel set with $\Delta(\omega):=\tau>0$. Then there is a constant $c$, depending only on $C,\Theta,\tau$, such that $V_K^\Theta(\omega)<c$.
\end{lemma}

\begin{proof}
Let $x\in\nu_K^{-1}(\omega)$. There is a supporting hyperplane $H$ of $K$ at $x$ with $u\in\omega$. Since $b(K)=1$, the hyperplane $H$ meets the ball $B^n$. Therefore, there is a constant $R$, depending only on $C$ and $\tau$, such that $H\cap C\subset RB^n$. Thus, $\nu_K^{-1}(\omega)\subset RB^n$. This implies that 
$$V_K^\Theta(\omega)=\Ha_\Theta^n(M(K,\omega)) < \Ha_\Theta^n(RB^n)=:c,$$
where $c<\infty$ depends only on $C,\tau,\Theta$.
\end{proof}

Let $j\in\N$. We have $V_K^\Theta(\omega_1) =\varphi(\omega_1)>0$. Since $b(b(K_j)^{-1}K_j)=1$, we get, with a constant $c>0$ depending only on $C,\Theta,\omega_1$,
$$ c> V_{b(K_j)^{-1}K_j}^\Theta(\omega_1) = b(K_j)^{-(n-q)}V_{K_j}^\Theta(\omega_1)$$
and hence
$$ b(K_j) >\left(\frac{\varphi(\omega_1)}{c}\right)^{\frac{1}{n-q}},$$
where the right-hand side is independent of $j$.

\vspace{2mm}

\noindent{\bf Remark.} Lemma \ref{L5.2} is modeled after Lemma 9 in \cite{Sch24a}. The latter lemma can also be applied in the proof of Theorem 5 in \cite{Sch18} and in the proof of Theorem 1 in \cite{Sch21}, where it is not mentioned. However, the application in the second part of the proof of Theorem 1 in \cite{Sch24c} is a mistake, since there $S_{n-1}^\Theta$ is homogeneous of negative degree in its first argument.

\vspace{2mm}

To show the right-hand inequality of (\ref{5.1}), we note that the measure $\Ha_\Theta^n$, defined by (\ref{1.0}), is homogeneous of degree $n-q>0$, hence we can choose a number $c_2>0$ with $\Ha_\Theta^n(c_2B^n)>\varphi(\Omega_{C^\circ})$. Suppose that $b(K_j)\ge c_2$ for some $j$. Then $C\cap c_2B^n\subset C\setminus K_j$, hence we obtain
$$ \varphi(\omega_j) = V_{K_j}^\Theta(\omega_j) = V_{K_j}^\Theta(\Omega_{C^\circ})= \Ha_\Theta^n(C\setminus K_j)>\Ha_\Theta^n(c_2B^n) >\varphi(\omega_j).$$
This is a contradiction. 

The following is Lemma 7 from \cite{Sch24a} (slightly reformulated).

\begin{lemma}\label{L5.3}
If $K\in ps(C)$ and $\omega\in\B(\Omega_{C^\circ})$, then
$$ \nu_K^{-1}(\omega) \subset \frac{b(K)}{\sin\Delta(\omega)}B^n.$$
\end{lemma}

To complete now the proof of Theorem \ref{T1.2}, we combine some previous arguments. Precisely as in the proof of Theorem 5 in \cite{Sch18}, one shows that there are a pseudo-cone $K\in ps(C)$ and an increasing sequence $(t_k)_{k\in\N}$ with $t_k\uparrow\infty$ such that, after a change of notation, the pseudo-cones $K_j\in \K(C,\omega_j)$ satisfy $V_{K_j}^\Theta= \varphi_j$ and 
$$ \emptyset\not= K_j\cap C^-(t_k) \to K\cap C^-(t_k) \quad\mbox{as }j\to\infty,\,\mbox{for each }k\in\N.$$

We fix a number $i\in \N$. By Lemma \ref{L5.3} and the boundedness of $(b_{K_j})_{j\in\N}$ there is a number $m\in \N$ such that
$$ \nu_{K_j}^{-1}(\omega_i)\subset C^-(t_m)\quad\mbox{for } j\ge i.$$
For $j\ge i$, the restrictions to $\omega_i$ (denoted by $\cdot\fed\omega_i$) satisfy
$$ \varphi\fed \omega_i= \varphi_j\fed\omega_i = V_{K_j}^\Theta\fed\omega_i =V_{K_j\cap C^-(t_m)}^\Theta\fed\omega_i.$$
From Lemma \ref{L3.7} and $K_j\cap C^-(t_m)\to K\cap C^-(t_m)$ we get
$$ V_{K_j\cap C^-(t_m)}^\Theta\fed \omega_i \to V_{K\cap C^-(t_m)}^\Theta\fed\omega_i\quad\mbox{weakly}$$
and therefore
$$ \varphi\fed\omega_i = V_{K\cap C^-{(t_m)}}^\Theta\fed\omega_i = V_K^\Theta\fed\omega_i.$$
Since this holds for all $i\in\N$ and $\bigcup_{i\in\N} \omega_i=\Omega_{C^\circ}$, we obtain $\varphi= V_K^\Theta$, which completes the proof.

\noindent Author's address:\\[2mm]
Rolf Schneider\\Mathematisches Institut, Albert--Ludwigs-Universit{\"a}t\\D-79104 Freiburg i.~Br., Germany\\E-mail: rolf.schneider@math.uni-freiburg.de

\end{document}